\documentclass[a4paper,11pt]{article}
\usepackage{authblk}
\usepackage{mathrsfs}
\usepackage{graphics}
\usepackage{latexsym}
\usepackage{cite}
\usepackage{graphicx,psfrag}
\usepackage{amsmath,amssymb,amsthm,mathtools}
\usepackage{multirow}
\usepackage{enumerate}
\usepackage{mathtools} 
\usepackage[bf, small]{titlesec}
\usepackage{relsize}

\newtheorem{Thm}{Theorem}[section]

\newtheorem{Lem}[Thm]{Lemma}

\title{A matrix equation $X^n = aI$}
\author[1]{Taehyeok HEO}
\author[1]{Jihoon CHOI\thanks{Corresponding author: gaouls@snu.ac.kr}}
\author[1]{Suh-Ryung KIM}
\affil[1]{Department of Mathematics Education,
Seoul National University, Seoul 151-742, Republic of Korea} 
\begin{document}
\date{}
\maketitle
\begin{abstract}
In this paper, we study a matrix equation $X^n = aI$. We factorize $X^n - aI$ based upon the factorization of $x^n - a$ and then give a necessary and sufficient condition for one of the factors to be the zero matrix.
\end{abstract}



\section{Introduction}

A polynomial $x^n - a$ with $n \ge 2$ can be factored into
\[ (x - a^\frac{1}{n})(x^{n-1} + a^\frac{1}{n}x^{n-2} + \cdots + a^\frac{n-2}{n} x + a^\frac{n-1}{n}) \]
if $a \ge 0$ or $n$ is odd. For the same reason, a matrix polynomial $X^n - aI$ with $n \ge 2$ can be factored into
\[ (X - a^\frac{1}{n}I)(X^{n-1} + a^\frac{1}{n}X^{n-2} + \cdots + a^\frac{n-2}{n} X + a^\frac{n-1}{n} I) \]
if $a \ge 0$ or $n$ is odd. From the factorization of $x^n-a$, we know that any root of $x^n-a = 0$ satisfies $x - a^\frac{1}{n} = 0$ or $x^{n-1} + a^\frac{1}{n}x^{n-2} + \cdots + a^\frac{n-2}{n} x + a^\frac{n-1}{n} = 0$. Though the ring $M_k(\mathbb{R})$ is not an integral domain, it is still interesting to ask for which $k,n,a$ the same situation occurs, that is, $X^n - aI = O$ and $X \neq a^\frac{1}{n}I$ imply
\[ X^{n-1} + a^\frac{1}{n}X^{n-2} + \cdots + a^\frac{n-2}{n} X + a^\frac{n-1}{n} I = O. \] Motivated by this question, we will study the sentence
\begin{equation}\label{eqn:sentence}
(\forall X \in M_k(\mathbb{R}))\left[X^n = aI \wedge X \neq a^\frac{1}{n}I \Rightarrow X^{n-1} + a^\frac{1}{n}X^{n-2} + \cdots + a^\frac{n-2}{n} X + a^\frac{n-1}{n} I = O\right]
\end{equation}
to obtain the following theorem.

\begin{Thm}\label{thm:char}
For integers $k,n$ $(k,n \ge 2)$ and $a \in \mathbb{R}$ satisfying the property that if $a < 0$, then $n$ is odd, the sentence
\[ (\forall X \in M_k(\mathbb{R}))\left[X^n = aI \wedge X \neq a^\frac{1}{n}I \Rightarrow X^{n-1} + a^\frac{1}{n}X^{n-2} + \cdots + a^\frac{n-2}{n} X + a^\frac{n-1}{n} I = O\right] \]
becomes true if and only if one of the following holds
\begin{enumerate}
\item[(i)] $a \neq 0$, $k = 2$, and $n$ is odd
\item[(ii)] $a=0$ and $n \ge k+1$.
\end{enumerate}
\end{Thm}

Suppose that $n$ is even and $a<0$. Then $x^n-a$ cannot have a linear factor over $\mathbb{R}$, so it is not meaningful to consider the sentence (\ref{eqn:sentence}) for the matrix equation $X^n - aI = O$ if $n$ is even and $a<0$. However, the polynomial $x^n -a$ can be factored into
\[ (x+(-a)^\frac{1}{n}\zeta)(x+ (-a)^\frac{1}{n}\zeta^3)\cdots(x+ (-a)^\frac{1}{n}\zeta^{2n-1}) = \prod_{i=1}^{n/2} \left( x^2 + (-a)^\frac{1}{n} \cos \frac{(2i-1)\pi}{n} x + (-a)^\frac{2}{n} \right) \]
where $\zeta = \exp(\frac{\pi i}{n})$. For the same reason, a matrix polynomial $X^n - aI$ can be factored into
\[ \prod_{i=1}^{n/2} \left( X^2 + (-a)^\frac{1}{n} \cos \frac{(2i-1)\pi}{n} X + (-a)^\frac{2}{n}I \right) \]
if $n$ is even and $a<0$. In the same context as the case where $n \ge 2$ and $a \ge 0$, or $n$ is odd, we may ask for which $k,n,a$ $X^n - aI = O$ implies
\[ X^2 + (-a)^\frac{1}{n} \cos \frac{(2i-1)\pi}{n} X + (-a)^\frac{2}{n}I = O \]
for some $i \in \{1,2,\ldots,\frac{n}{2}\}$. Based on this question, if $n$ is even and $a<0$, we will study the sentence
\begin{equation}\label{eqn:sentence2}
(\forall X \in M_k(\mathbb{R}))\left[ X^n = aI \Rightarrow \left( \exists i \in \left\{1,2,\ldots,\frac{n}{2} \right\} \right) \left[X^2 + (-a)^\frac{1}{n} \cos \frac{(2i-1)\pi}{n} X + (-a)^\frac{2}{n}I = O\right] \right]
\end{equation}
to present the following theorem.

\begin{Thm}\label{thm:char2}
For integers $k,n$ $(k,n \ge 2)$ with $n$ is even and $a < 0$, the sentence
\[ (\forall X \in M_k(\mathbb{R}))\left[ X^n = aI \Rightarrow \left( \exists i \in \left\{1,2,\ldots,\frac{n}{2} \right\} \right) \left[X^2 + (-a)^\frac{1}{n} \cos \frac{(2i-1)\pi}{n} X + (-a)^\frac{2}{n}I = O\right] \right] \]
becomes true if and only if $k$ is odd, or $k$ is even and $n=2$.
\end{Thm}

We will prove Theorem~\ref{thm:char} in Section~2 and Theorem~\ref{thm:char2} in Section~3. For undefined terms, the reader may refer to \cite{HJ}.


\section{A proof of Theorem~\ref{thm:char}}

We take a matrix $A$ with real entries. We call $A$ a {\it non-simple $n$th root of $aI$} if it satisfies
\[  A^n = aI \mbox{ and } A \neq a^\frac{1}{n}I.  \]

We first show that Theorem~\ref{thm:char} holds for $a=0$.  

To show the `only if' part, we consider the case $a = 0$ and $n \le k$. We define the matrix $A = (a_{ij})$ by
\[ a_{ij} = \begin{cases} 1 & \mbox{if } j = i+k-n+1; \\ 0 & \mbox{otherwise}. \end{cases} \]
See the matrix below for an illustration for $n=k$:
\[ A = \begin{pmatrix} 0 & 1 & 0 & 0 & \cdots & 0 & 0 \\ 0 & 0 & 1 & 0 & \cdots & 0 & 0 \\ 0 & 0 & 0 & 1 & \cdots & 0 & 0 \\ \vdots & \vdots & \vdots & \vdots & \ddots & \vdots & \vdots \\ 0 & 0 & 0 & 0 & \cdots & 0 & 1 \\ 0 & 0 & 0 & 0 & \cdots & 0 & 0 \end{pmatrix}. \]
It can easily be checked that $A \neq O, A^n=O$ but $A^{n-1} \neq O$. Therefore the `only if' part of Theorem~\ref{thm:char} follows if $a=0$.

To show the `if' part, we consider the case $a = 0$ and $n \ge k+1$. Suppose that $A^n = O$. Then $0$ is the only eigenvalue of $A$ and so the Jordan matrix of $A$ is of the form
\[ J_A = \begin{pmatrix} J_{n_1}(0) & O & O  & \cdots & O & O \\ O & J_{n_2}(0) & O & \cdots & O & O \\ O & O & J_{n_3}(0) & \cdots & O & O \\ \vdots & \vdots & \vdots & \ddots & \vdots & \vdots \\
O & O & O & \cdots & J_{n_{l-1}}(0) & O \\ O & O & O & \cdots & O & J_{n_l}(0) \end{pmatrix} \]
where $n_1 + \cdots + n_l = k$ and $J_{n_i}(0)$ is the Jordan block of order $n_i$ with eigenvalue $0$. Since $J_A$ is a strictly upper triangular matrix of order $k$, it is true that $(J_A)^k = O$. Because $A$ is similar to $J_A$, $A^k = O$. Then $A^{n-1} = A^k A^{n-k-1} = O A^{n-k-1} = O$. Therefore the `if' part of Theorem~\ref{thm:char} follows if $a=0$.

Now we show Theorem~\ref{thm:char} when $a \neq 0$. If $a > 0$, then
\begin{equation}\label{simple}
A^n=aI \Leftrightarrow \left(a^{-\frac{1}{n}} A \right)^n = I.
\end{equation}
Therefore, by substituting $a^\frac{1}{n} X$ or $a^{-\frac{1}{n}} X$ into $X$, it is sufficient to consider the case $a=1$ if $a>0$. Suppose $a<0$. Note that $X^n = aI = -(-a)I$. Then, since $-a>0$, by (\ref{simple}),
\begin{equation}\label{simple2}
A^n = aI \Leftrightarrow \left( (-a)^{-\frac{1}{n}} A \right)^n = -I
\end{equation} 
and therefore it is sufficient to consider the case $a = -1$.

We need the following lemma.

\begin{Lem}
Suppose that $\left( \begin{smallmatrix} p & q \\ 0 & r \end{smallmatrix} \right) \in M_2(\mathbb{C})$ is a non-simple $n$th root of $I$ for an integer $n \ge 2$. Then
\[ (p, r)=((\zeta_n)^u, (\zeta_n)^v) \]
for some $u, v \in \{ 0, 1, \ldots, n-1\}$. where $\zeta_n=\exp(2\pi i/n)$.
\label{lem:cyclotomic}
\end{Lem}
\begin{proof}
For notational convenience, let $A = \left( \begin{smallmatrix} p & q \\ 0 & r \end{smallmatrix} \right)$. We first prove by induction on $n$ that
\begin{equation} \label{eqn:comp}
A^n = (p^{n-1} + p^{n-2}r + \cdots + pr^{n-2}+r^{n-1})A -pr(p^{n-2} + \cdots + r^{n-2})I.
\end{equation}
The statement~(\ref{eqn:comp}) is true for $n=2$ by the Cayley-Hamilton Theorem. Suppose that (\ref{eqn:comp}) is true for $n$. Then, by the induction hypothesis,
\begin{align*}
A^{n+1} &= A^nA = (p^{n-1} + p^{n-2}r + \cdots + pr^{n-2}+r^{n-1})A^2 -pr(p^{n-2} + \cdots + r^{n-2})A \\
&= (p^{n-1} + p^{n-2}r + \cdots + pr^{n-2}+r^{n-1})((p+r)A - prI) -pr(p^{n-2} + \cdots + r^{n-2})A.
\end{align*}
By simplifying the right-hand side of the second equality, we can check that (\ref{eqn:comp}) is true for $n+1$.

Now, since $A$ is an $n$th root of $I$, by~(\ref{eqn:comp}),
\[ I = A^n = (p^{n-1}+p^{n-2}r + \cdots + pr^{n-2}+r^{n-1})A -pr(p^{n-2}+p^{n-3}r + \cdots + pr^{n-3}+r^{n-2})I \]
and, by comparing $(1,2)$ and $(2,2)$ entries of the matrix on the right with those of $I$ on the left, we obtain the system of equations
\[ q \left(p^{n-1}+p^{n-2}r + \cdots + pr^{n-2}+r^{n-1}\right) = 0,\]
and
\[ 1 = (p^{n-1} + p^{n-2}r + \cdots + pr^{n-2}+r^{n-1})r -  pr(p^{n-2}+p^{n-3}r + \cdots + pr^{n-3}+r^{n-2}) \]
or
\[ 1 = r^n. \]
If $q=0$, then $I = A^n= \left( \begin{smallmatrix} p^n & 0 \\ 0 & r^n \end{smallmatrix} \right)$ and so the lemma follows. If $q \neq 0$, then by solving this system, we have
\[ (p, r)=((\zeta_n)^u, (\zeta_n)^v) \]
for some $u, v \in \{ 0, 1, \ldots, n-1\}$.
\end{proof}

To show the `if' part, take a non-simple $n$th root $A \in M_2(\mathbb{R})$ of $I$. Since $n$ is odd, 
\begin{equation} \label{eqn:factorization}
A^{n-1} + \cdots + A + I = \prod_{w=1}^{(n-1)/2} \left(A^2-2\cos \left( \frac{2\pi}{n} w \right) A + I \right).
\end{equation}
Let $J_A := \left( \begin{smallmatrix} p & q \\ 0 & r \end{smallmatrix} \right)$ be the Jordan matrix of $A$. Since $A$ is similar to $J_A$, $J_A$ is also a non-simple $n$th root of $I$. By Lemma~\ref{lem:cyclotomic}, 
\begin{equation} \label{eqn:pair}
(p, r)=((\zeta_n)^u, (\zeta_n)^v)
\end{equation}
for some $u, v \in \{ 0, 1, \ldots, n-1\}$. Moreover, by the similarity, $\det(A-\lambda I)=\det(J_A-\lambda I)$. Since $A \in M_2(\mathbb{R})$, $\det(A-\lambda I)$ is a polynomial over $\mathbb{R}$ and so is $\det(J_A - \lambda I)$. Then $p+r$ and $pr$ are in $\mathbb{R}$ and so $u+v = n$ for $u,v$ in (\ref{eqn:pair}). Therefore, by the symmetry of $u$ and $v$,
\begin{equation} \label{eqn:quadratic}
A^2-2\cos \left( \frac{2\pi}{n} u\right) A + I = O
\end{equation}
for some $u \in \{0,1,\ldots,\frac{n-1}{2}\}$. Suppose that $u=0$. Then $v=n$ and so $J_A = \left( \begin{smallmatrix} 1 & q \\ 0 & 1 \end{smallmatrix} \right)$. However, $(J_A)^n = \left( \begin{smallmatrix} 1 & nq \\ 0 & 1 \end{smallmatrix} \right)$ which cannot equal $I$ unless $J_A = I$, and we reach a contradiction to the fact that $J_A$ is a non-simple $n$th root of $I$. Thus $u \in \{1,2,\ldots,\frac{n-1}{2}\}$ in the statement (\ref{eqn:quadratic}) and so, by (\ref{eqn:factorization}), 
\[ A^{n-1} + \cdots + A + I = O. \]

Now take a non-simple $n$th root $A \in M_2(\mathbb{R})$ of $-I$. If $n$ is odd, then $-A$ is a non-simple $n$th root of $I$ and so
\[ O = (-A)^{n-1} + (-A)^{n-2} + \cdots + (-A) + I = A^{n-1} - A^{n-2} + \cdots - A + I. \]
Hence we have shown that the `if' part of Theorem~\ref{thm:char} is true when $a \neq 0$.

It remains to show the `only if' part, that is, the sentence (\ref{eqn:sentence}) is not true if either $a \neq 0$ and $k \ge 3$, or $a \neq 0$ and $n$ is even. We will give a counterexample for each of the following cases:
\begin{table}[h]
\begin{center}
\begin{tabular}{c|c|c}
& $a=1$ & $a=-1$ \\ \hline
$n$ is even, $k$ is even & (i) &  \\ \hline
$n$ is even, $k$ is odd & (ii) &  \\ \hline
$n$ is odd, $k$ is even $(k \ge 3)$ & (iii) & (v) \\ \hline
$n$ is odd, $k$ is odd & (iv) & (vi)
\end{tabular}
\end{center}
\end{table}

We denote the matrix $\left( \begin{smallmatrix} 0 & 1 \\ 1 & 0 \end{smallmatrix} \right)$ by $T$, the zero matrix of order two by $O_2$, and the rotation matrix
\[ \begin{pmatrix} \cos {\theta} & -\sin {\theta} \\ \sin {\theta} & \cos {\theta}  \end{pmatrix} \]
by $R_\theta$. In addition, we distinguish identity matrices by denoting the identity matrix of order $l$ by $I_l$.

\noindent(i) {\it $n$ is even and $k$ is even.} We take the matrix of order $k$
\[ A := \begin{pmatrix} T & O_2 & \cdots & O_2 \\ O_2 & T & \cdots & O_2 \\ \vdots & \vdots & \ddots & \vdots \\ O_2 & O_2 & \cdots & T \end{pmatrix} \]
By block multiplication,
\[ A^n = \begin{pmatrix} T^n & O_2 & \cdots & O_2 \\ O_2 & T^n & \cdots & O_2 \\ \vdots & \vdots & \ddots & \vdots \\ O_2 & O_2 & \cdots & T^n \end{pmatrix} = I_k \]
as an even power of $T$ is the identity matrix of order two. Since all of the diagonal entries of $A$ are zero, obviously $A \neq I_k$. However,
\begin{eqnarray*}
&& A^{n-1} + a^\frac{1}{n}A^{n-2} + a^\frac{2}{n}A^{n-3} + a^\frac{3}{n}A^{n-4} + \cdots + a^\frac{n-2}{n} A + a^\frac{n-1}{n} I_k \\
&=& A^{n-1} + A^{n-2} + A^{n-3} + A^{n-4} + \cdots + A + I_k \\
&=& A + I_k + A + I_k + \cdots + A + I_k \\
&=& \frac{n}{2}(A+I_k) \neq O,
\end{eqnarray*}
so $A$ is a counterexample to the sentence (\ref{eqn:sentence}).

\noindent(ii) {\it $n$ is even and $k$ is odd.} We take the matrix of order $k$
\[ \begin{pmatrix} 1 & 0 & 0 & \cdots & 0 \\ 0 & T & O_2 & \cdots & O_2 \\ 0 & O_2 & T & \cdots & O_2 \\ \vdots & \vdots & \vdots & \ddots & \vdots \\ 0 & O & O & \cdots & T \end{pmatrix}. \]
By applying a similar argument for the case (i), we may show that the given matrix is a counterexample to the sentence (\ref{eqn:sentence}).

\noindent(iii) {\it $n$ is odd and $k$ is even $(k \ge 3)$.} We take the matrix of order $k$
\[ A:=\begin{pmatrix} I_2 & O_2 & O_2 & \cdots & O_2 \\ O_2 & R_{2\pi/n} & O_2 & \cdots & O_2 \\ O_2 & O_2 & R_{2\pi/n} & \cdots & O_2 \\ \vdots & \vdots & \vdots & \ddots & \vdots \\ O_2 & O_2 & O_2 & \cdots & R_{2\pi/n} \end{pmatrix}. \]
By block multiplication,
\[ A^n = \begin{pmatrix} (I_2)^n & O_2 & \cdots & O_2 \\ O_2 & (R_{2\pi/n})^n & \cdots & O_2 \\ \vdots & \vdots & \ddots & \vdots \\ O_2 & O_2 & \cdots & (R_{2\pi/n})^n \end{pmatrix} = I_k \]
as the $n$th power of $R_{2\pi/n}$ is the identity matrix of order two. Since $k \ge 3$, $(3,3)$ entry of $A$ exists and, by the hypothesis that $n \ge 3$, the $(3,3)$ entry of $A$ is not equal to $1$. However, the $(1,1)$ entry of $A$ is $1$, so $A \neq I_k$. Moreover, the $(1,1)$ entry of $A^i$ equals $1$ for any nonnegative integer $i$, so the $(1,1)$ entry of $A^{n-1} + a^\frac{1}{n}A^{n-2} + \cdots + a^\frac{n-2}{n} A + a^\frac{n-1}{n} I_k $ cannot be zero. Thus $A^{n-1} + a^\frac{1}{n}A^{n-2} + \cdots + a^\frac{n-2}{n} A + a^\frac{n-1}{n} I_k \neq O$ and so $A$ is a counterexample to the sentence (\ref{eqn:sentence}).

\noindent(iv) {\it $n$ is odd and $k$ is odd $(k \ge 3)$.} We take the matrix of order $k$
\[ \begin{pmatrix} 1 & 0 & 0 & \cdots & 0 \\ 0 & R_{2\pi/n} & O_2 & \cdots & O_2 \\ 0& O_2 & R_{2\pi/n} & \cdots & O_2 \\ \vdots & \vdots & \vdots & \ddots & \vdots \\ 0 & O_2 & O_2 & \cdots & R_{2\pi/n} \end{pmatrix}. \]
By applying a similar argument for the case (iii), we may show that the given matrix is a counterexample to the sentence (\ref{eqn:sentence}).

\noindent(v) {\it $n$ is odd and $k$ is even $(k \ge 3)$.} We take the matrix of order $k$
\[ A := \begin{pmatrix} -I_2 & O_2 & O_2 & \cdots & O_2 \\ O_2 & -R_{2\pi/n} & O_2 & \cdots & O_2 \\ O_2 & O_2 & -R_{2\pi/n} & \cdots & O_2 \\ \vdots & \vdots & \vdots & \ddots & \vdots \\ O_2 & O_2 & O_2 & \cdots & -R_{2\pi/n} \end{pmatrix}. \]
By block multiplication,
\[ A^n = \begin{pmatrix} (-I_2)^n & O_2 & \cdots & O_2 \\ O_2 & (-R_{2\pi/n})^n & \cdots & O_2 \\ \vdots & \vdots & \ddots & \vdots \\ O_2 & O_2 & \cdots & (-R_{2\pi/n})^n \end{pmatrix} = -I_k \]
as the $n$th power of $-R_{2\pi/n}$ equals $-I_2$. Since $k \ge 3$, $(3,3)$ entry of $A$ exists and, by the hypothesis that $n \ge 3$, the $(3,3)$ entry of $A$ is not equal to $-1$. However, the $(1,1)$ entry of $A$ is $-1$, so $A \neq -I$. However,
\begin{eqnarray*}
&& A^{n-1} + a^\frac{1}{n}A^{n-2} + a^\frac{2}{n}A^{n-3} + a^\frac{3}{n}A^{n-4} + \cdots + a^\frac{n-2}{n} A + a^\frac{n-1}{n} I_k \\
&=& A^{n-1} - A^{n-2} + A^{n-3} - A^{n-4} + \cdots - A + I_k.
\end{eqnarray*}
Now, the $(1,1)$ entry of $A^i$ equals 1 if $i$ is even and $-1$ if $i$ is odd. Therefore the $(1,1)$ entry of $A^{n-1} - A^{n-2} + A^{n-3} - A^{n-4} + \cdots - A + I_k$ equals $n$ and so $A^{n-1} - A^{n-2} + A^{n-3} - A^{n-4} + \cdots - A + I_k \neq O$. Hence $A$ is a counterexample to the sentence (\ref{eqn:sentence}).

\noindent(vi) {\it $n$ is odd and $k$ is odd.} We take the matrix of order $k$
\[ \begin{pmatrix} -1 & 0 & 0 & \cdots & 0 \\ 0 & -R_{2\pi/n} & O_2 & \cdots & O_2 \\ 0& O_2 & -R_{2\pi/n} & \cdots & O_2 \\ \vdots & \vdots & \vdots & \ddots & \vdots \\ 0 & O_2 & O_2 & \cdots & -R_{2\pi/n} \end{pmatrix}. \]
By applying a similar argument for the case (v), we may show that the given matrix is a counterexample to the sentence (\ref{eqn:sentence}). Hence we have shown the `only if' part of Theorem~\ref{thm:char} and the proof of Theorem~\ref{thm:char} is complete.


\section{A proof of Theorem~\ref{thm:char2}}

In this section, it is assumed that $n$ is even and $a<0$. By (\ref{simple2}), it is sufficient to consider the case $a = -1$.

First we show the `if' part of Theorem~\ref{thm:char2}.

Suppose that $k$ is odd. We will show that there is no matrix whose $n$th power equals $-I$.
Assume, to the contrary, that there exists $A \in M_k(\mathbb{R})$ such that $A^n = -I$. Then, for the Jordan matrix $J_A$ of $A$, the following holds:
\begin{equation} \label{eqn:negative}
(J_A)^n=-I
\end{equation}
and
\begin{equation} \label{eqn:det}
\det(A-\lambda I)=\det(J_A - \lambda I).
\end{equation}
We denote the $(j,j)$ entry of $J_A$ by $a_j$ for each $j=1, 2, \ldots, k$. Since $J_A$ is upper triangular, taking the $n$th power of $J_A$ gives diagonal elements the $n$th power of diagonal elements of $J_A$. By (\ref{eqn:negative}), $a_j^n=-1$. Since $A \in M_k(\mathbb{R})$, $\det(A-\lambda I)$ is a polynomial in $\lambda$ with real coefficients and so is $\det(J_A - \lambda I)$ by (\ref{eqn:det}).
Therefore the constant term $-a_1a_2 \cdots a_k$ is real. On the other hand, since $k$ is odd,
\[ (a_1a_2 \cdots a_k)^n = (a_1)^n (a_2)^n \cdots (a_k)^n = (-1)^k = -1. \]
However, since $n$ is even, there is no real $a_1a_2 \cdots a_k$ satisfying the last equality and we reach a contradiction. Hence there is no matrix whose $n$th power equals $-I$ and the `if' part is vacuously true if $k$ is odd.

Now suppose that $k$ is even and $n = 2$. Then the sentence (\ref{eqn:sentence2}) becomes
\[ (\forall X \in M_k(\mathbb{R}))\left[ X^2 = -I \Rightarrow X^2 + I = O \right], \]
which is trivially true. Hence the `if' part holds.

We show the `only if' part by giving a counterexample to the sentence (\ref{eqn:sentence2}) when $k$ is even and $n \ge 4$. For a notational convenience, we denote
\[ \begin{pmatrix} \displaystyle \cos\frac{(2j-1)\pi}{n} & \displaystyle -\sin \frac{(2j-1)\pi}{n} \\ \displaystyle \sin\frac{(2j-1)\pi}{n} & \displaystyle \cos\frac{(2j-1)\pi}{n} \end{pmatrix}. \]
by $R_j$ instead of $R_{{(2j-1)\pi}/{n}}$. Now we take the following matrix
\[ A = \begin{pmatrix} R_1 & O & O & \cdots & O \\ O & R_2 & O & \cdots & O \\ O & O & R_2 & \cdots & O \\ \vdots & \vdots & \vdots & \ddots & \vdots \\ O & O & O & \cdots & R_2 \end{pmatrix}. \]
Since $(R_j)^n = -I$ for $j=1,2,\ldots,\frac{n}{2}$, $A^n = -I$.

Take any $i \in \{ 1,2,\ldots,\frac{n}{2} \}$. By the Cayley-Hamilton Theorem, $$(R_j)^2-2\cos\frac{(2j-1)\pi}{n}R_j+I=O$$ for each $j=1,2,\ldots,\frac{n}{2}$. Then, if $i=1$, 
$$(R_2)^2-2\cos\frac{\pi}{n}R_2+I = \left( -2\cos\frac{\pi}{n} + 2\cos\frac{3\pi}{n} \right) R_2 \neq O$$
and so $A^2-2\cos\frac{\pi}{n}A+I \neq O$. If $i \neq 1$, then 
$$(R_1)^2-2\cos\frac{(2i-1)\pi}{n}R_1+I = \left( -2\cos\frac{(2i-1)\pi}{n} + 2\cos\frac{\pi}{n} \right) R_1 \neq O$$
and so and so $A^2-2\cos\frac{(2i-1)\pi}{n}A+I \neq O$. Thus $A$ is a counterexample to the sentence (\ref{eqn:sentence2}) and we complete the proof of Theorem~\ref{thm:char2}.

\section{Closing remarks}
We may consider the complex number version of Sentence (\ref{eqn:sentence})
\[
(\forall X \in M_k(\mathbb{C}))\left[X^n = aI \wedge X \neq a^\frac{1}{n}I \Rightarrow X^{n-1} + a^\frac{1}{n}X^{n-2} + \cdots + a^\frac{n-2}{n} X + a^\frac{n-1}{n} I = O\right]
\]
for integers $k, n$ with $k,n \ge 2$ and $a \in \mathbb{C}$. However, it cannot happen except the case $a=0$ and $n \ge k+1$. If $a=0$, then the same argument for the real number case is applied. If $a \neq 0$, then the matrix
\[\begin{pmatrix} a^{\frac{1}{n}} & 0 & 0 & \cdots & 0 \\ 0 & a^{\frac{1}{n}}\zeta_n & 0 & \cdots & 0 \\ 0 & 0 & a^{\frac{1}{n}}\zeta_n & \cdots & 0 \\ \vdots & \vdots & \vdots & \ddots & \vdots \\ 0 & 0 & 0 & \cdots & a^{\frac{1}{n}}\zeta_n \end{pmatrix} \]
becomes a counterexample when $a^\frac{1}{n}$ is a number satisfying $z^n = a$ and  $\zeta_n = \exp(\frac{2\pi i}{n})$.


\begin{thebibliography}{999}%

\bibitem{HJ}
{R.\ A.\ Horn and C. R. Johnson}:
{\it Matrix Analysis},  Cambridge,  2013.


\end{thebibliography}
\end{document}